\documentclass[12pt,reqno]{amsart}
\usepackage{amsaddr}
\usepackage{amssymb,latexsym,amsmath,epsfig,amsthm,mathrsfs}
\usepackage{rotating}
\usepackage{graphicx}
\usepackage{amssymb}
\usepackage{lineno}
\usepackage{enumitem}
\usepackage{cite}
\usepackage[top=3cm, bottom=3cm, left=3cm, right=3cm]{geometry}
\usepackage[usenames]{color}
\usepackage[colorlinks=true,
linkcolor=blue,
filecolor=blue,
citecolor=blue]{hyperref}

\usepackage{mathtools}

\makeatletter
\DeclarePairedDelimiterX{\pmodx}[1]{(}{)}{{\operator@font mod}\mkern6mu#1}
\renewcommand{\pmod}{%
	\allowbreak
	\if@display\mkern18mu\else\mkern8mu\fi
	\pmodx
}
\makeatother

\newtheorem{theorem}{Theorem}[section]

\theoremstyle{definition}

\theoremstyle{remark}

\theoremstyle{theorem}

\theoremstyle{definition}

\numberwithin{equation}{section}

\begin{document}
	
\title[$S_h$-sets in abelian groups]{$S_h$-sets in abelian groups}
	
%%%%%%    Information for first author
	
\author[V Goswami]{Vivekanand Goswami}
\address{\em{\small Department of Mathematics, Indian Institute of Technology Bhilai, Durg – 491001, Chhattisgarh, India\\
email: vivekanandg@iitbhilai.ac.in}}
%\email{vivekanandg@iitbhilai.ac.in}
	
%%%%%%   Information for second author
	
\author[R K Mistri]{Raj Kumar Mistri$^{*}$}
\address{\em{\small Department of Mathematics, Indian Institute of Technology Bhilai, Durg – 491001, Chhattisgarh, India\\
email: rkmistri@iitbhilai.ac.in}}
	
\thanks{$^{*}$Corresponding author}

%    General info


\begin{abstract}
For a positive integer $h$, a subset $A = \{a_1, \dots, a_k\}$ of an additive abelian group $G$ is called an $S_h$-set of size $k$ if all sums of $h$ distinct elements in $S$ are distinct. For fixed positive integers $h$ and $k$, let $v_h(k)$ denote the order of the smallest abelian group containing an $S_h$-set of size $k$. A lower bound for $v_2(k)$ is known. In this paper, we establish a lower bound for $v_3(k)$. Using our argument for $h=2$, we recover the known bound for $v_2(k)$.
\end{abstract}

%%% ----------------------------------------------------------------------
\maketitle
%%% ----------------------------------------------------------------------
	
%\tableofcontents

\section{Introduction}
Let $G$ be an additive abelian group. Let $h$ be a positive integer, and let $A = \{a_1, \dots, a_k\}$ be a subset of $G$. The cardinality of a set $A$ is denoted by $|A|$. The \emph{restricted $h$-fold sumset} of $A$, denoted by $h^\wedge A$, is defined as
\[h^\wedge A = \{a_1 + a_2 + \cdots + a_h : a_i \in A, a_i \neq a_j \ \text{for} \ i \neq j\}.\]
A subset $A$ of $G$ is called an $S_h$-set of size $k$ if all sums of $h$ distinct elements in $S$ are distinct \big(that is, $|h^\wedge A| = \binom{k}{h}$\big). Let $s_h(G)$ denote the largest cardinality of an $S_h$-set in $G$, and let $v_h(k)$ denote the order of the smallest abelian group $G$ for which $s_h(G)\geq k$.
	 
The study of $v_h(k)$ and $S_h$-sets has applications in coding theory (see \cite{graham1980a, graham1980b, guy1994}). In particular, $S_h$-sets are closely related to constant weight error-correcting codes. A constant weight code of length $k$ and weight $w$ is a set of binary vectors of length $k$, each having weight $w$, such that the Hamming distance between any two distinct vectors is at least $2d$. The maximum size of such a code is denoted by $A(k,2d,w)$. A connection between these parameters is given in \cite{graham1980a, brouwer1990}. In \cite[Theorem 16]{brouwer1990} it is shown that
\[
A(k,2d,w)\geq \frac{1}{v_{d -1}(k)}\binom{k}{w}.
\]	 
	 
In \cite{haanpaa2004, haanpaa2007}, bounds for $s_2(G)$ and $v_2(k)$ were established.  Recently, $S_h$-sets were studied in the setting of finite vector spaces over finite fields in \cite{pantoja2025}. We refer the reader to \cite{graham1980a, graham1980b, guy1994, bajnok2018} for open problems related to $S_h$-sets.

The following theorem gives a lower bound for $v_2(k)$. Although it is not stated explicitly in \cite{haanpaa2007}, it follows from the proof of the corresponding result therein.

\begin{theorem}[{\cite[Theorem 2]{haanpaa2007}}]\label{s2}
Let  $A$ be a $k$-element $S_2$-set in a finite abelian group $G$. Then 
\[|G| \ge k(k-3) - |G[2]| + 1,\]
where $G[2] = \{x\in G \setminus \{0\} : 2x = 0\}$.
\end{theorem}		
				
In this paper, we establish a lower bound for $v_3(k)$. More precisely, we prove the following theorem.
		
\begin{theorem}\label{s3}
 Let $k$ be an integer with $k \geq 5$. Let  $A$ be a $k$-element $S_3$-set in a finite abelian group $G$. Then 
\[|G|  \geq \frac{k(k-1)(k-2)}{2} - \frac{7k(k - 1)}{2} - 2k - |G[2]|(k - 2),\]
where $G[2] = \{x\in G \setminus \{0\} : 2x = 0\}$.
\end{theorem}

We prove Theorem \ref{s3} in Section \ref{ps3}. Theorem \ref{s2} was established in \cite{haanpaa2007}. We present another proof of Theorem \ref{s2} in Section \ref{ps2}, that relies on different combinatorial arguments.		
	
\section{Notation} \label{sec-notation}

Let $A=\{a_1,\ldots,a_k\}$ be an $S_h$-set in $G$, and fix an ordering
\[
a_1\prec a_2\prec\cdots\prec a_k
\]
of its elements. Thus, for any $x,y\in A$, where $x=a_i$ and $y=a_j$, we have
\[
x\prec y \quad\text{if and only if}\quad i<j.
\]
 We define 
\[A_i = A \setminus \{a_i\} ~ \text{for $i = 1$ to $k$}\] 
and
\[A \ominus A = \{b - a : a, b \in A ~ \text{and}~ a \prec b\}\]
For a subset $S$ of $G$, we define 
\[S[2] = \{x \in S \setminus \{0\} : 2x = 0\}\]
and
\[2\cdot S = \{2x : x \in S\}.\]
Note that, an $S_h$-set is also an $S_r$-set for every positive integer $r$ with $r \leq h$ if the set has at least $2h - 1$ elements.
Throughout the paper, the symbols $a$, $b$, $c$, $d$, $x_i$, and $y_i$ denote elements of $A$.

\section{Proof of Theorem \ref{s2}}\label{ps2}
				
\begin{proof}[Proof of Theorem \ref{s2}] 
 Let
\[X_{2i} = A_i \times \{a_i\},\]
\[X_2 = \bigcup_{i = 1}^{k} X_{2i}\]
and for $g \in G$, we define
\[Y^2_g = \{(x_1, a_i) \in X_2 : x_1 - a_i = g\}.\] 
Note that, if $(x_1, a_i)$ and  $(y_1, a_j)$ are two distinct elements in  $Y^2_g$ then $i \neq j$. The smallest $i$ such that $(x_1, a_i) \in Y^2_g$ is denoted by $i_g$. 
Consider the following collection of sets: 
\[\mathscr{Y}_2 = \left\{\left\{(x_1, a_{i_g}), (y_1, a_j)\right\} \subseteq Y^2_g : g \in G\right\}.\]
We next consider two subsets $\mathscr{Y}_{21}$ and $\mathscr{Y}_{22}$ of $\mathscr{Y}_2$, defined as follows: 
\[\mathscr{Y}_{21} = \left\{\{(x_1, a_i), (y_1, a_j)\} \in \mathscr{Y}_2 : ~ \text{either}~ a_j = x_1, ~ a_i \neq y_1, ~\text{or} ~ a_j \neq  x_1, ~ a_i = y_1\right\},\]
and 
\[\mathscr{Y}_{22} = \left\{\{(x_1, a_i), (y_1, a_j)\} \in \mathscr{Y}_2 : a_j = x_1 ~ \text{and} ~ a_i = y_1\right\}.\]

Since $A$ is an  $S_2$-set, $\mathscr{Y}_{21}$ and $\mathscr{Y}_{22}$ partition $\mathscr{Y}_2$.

Clearly, $|Y_0^2| = 0$. There are at most $|G| - 1$ possible values of $g$ for which $|Y^2_g| \geq 1$. Therefore, total number of elements of the form $(x_1, a_{i_g}) \in X_2$ are at most $|G|-1$. The remaining elements of $X_2$ correspond to distinct elements of  $\mathscr{Y}_2$. So we have
\[|G| - 1 + |\mathscr{Y}_2| \geq |X_2| = k(k-1).\]
That is, 
\begin{equation}\label{Y2}
|G| - 1 + |\mathscr{Y}_{21}| + |\mathscr{Y}_{21}| \geq k(k-1).
\end{equation}
We first estimate a bound for $|\mathscr{Y}_{21}|$ and then for $|\mathscr{Y}_{22}|$.

Consider the set $T_{21}$ defined as:
\[T _{21} = \{(a, b, c) : 2a = b + c ~\text{and}~ a, b, c ~\text{all distinct}\}.\] 
We define a map $f_{21}:\mathscr{Y}_{21} \to T_{21}$ by
\[f\big(\{(a, b),(c, a)\}\big)=(a, b, c),\]
where $\{(a, b),(c, a)\}=\{(c, a),(a, b)\}$ both representations are assigned the same image $(a, b, c)$. Also, it is straightforward to verify that this map is a bijection.
			
As $A$ is an $S_2$-set, we have \[|T_{21}| \leq 2k.\]
That is,  
\begin{equation}\label{Y21}
|\mathscr{Y}_{21}| \leq 2k.
\end{equation}
Let $\{(a, b), (c, d)\}$ be a set in $\mathscr{Y}_{22}$. Since $a = d, b = c$, we have $a - b = b - a$, and $2(a - b) = 2(b - a) = 0$. 
Consider the set $T_{22}$ defined as:
\[T_{22} = \{(a, b) : 2(a - b) = 0~ \text{and} ~ b \prec a\}.\] 
We define a map $f_{22}:\mathscr{Y}_{22} \to T_{22}$ by
\[
f_2\big(\{(a, b),(b, a)\}\big)=
\begin{cases}
(a, b), & \text{if } b \prec a,\\
(b, a), & \text{if } a \prec b.
\end{cases}
\]
where $\{(a, b),(b, a)\}=\{(a, b),(b, a)\}$ both representations are assigned the same image $(a, b)$ or $(b, a)$. Also, it is straightforward to verify that this map also is a bijection.
			
Observe that, if $(a, b), (c, d) \in T_{22}$ such that $a - b = c- d$, then $(a, b) = (c, d)$. This implies that \[|T_{22}| = |(A \ominus A)[2]| \leq |G[2]|.\]
That is, 
\begin{equation}\label{Y22}
|\mathscr{Y}_{22}| \leq |G[2]|.
\end{equation}
Using \eqref{Y2}, \eqref{Y21} and  \eqref{Y22}, we get
\[|G| - 1 + 2k + |(A \ominus A)[2]| \geq k(k - 1).\]
This implies that	
\[|G| \ge k(k - 3) - |(A \ominus A)[2]| + 1.\]
Hence the Theorem \ref{s2} follows.
\end{proof}
		
\section{Proof of Theorem \ref{s3}}\label{ps3}

\begin{proof}[Proof of Theorem \ref{s3}] 
Let
\[X_{3i} = \big(2^\wedge A_i\big) \times \{a_i\},\]
\[X_3 = \bigcup_{i = 1}^{k} X_{3i}\]
and for $g \in G$, we define
\[Y^3_g = \{(x_1 + x_{2}, a_i) \in X_3 : x_1 + x_{2} - a_i = g\}.\] 
Note that, if $(x_1 + x_{2}, a_i)$ and  $(y_1 + y_{2}, a_j)$ are two distinct elements in  $Y^3_g$ then $i \neq j$. The smallest $i$ such that $(x_1 + x_{2}, a_i) \in Y^3_g$ is denoted by $i_g$. 
Consider the following collection of sets: 
\[\mathscr{Y}_3 = \left\{\left\{(x_1 + x_{2}, a_{i_g}), (y_1 + y_{2}, a_j)\right\} \subseteq Y^3_g : g \in G\right\}.\]

We next consider two subsets $\mathscr{Y}_{31}$ and $\mathscr{Y}_{32}$ of $\mathscr{Y}_3$, defined as follows: 
\begin{align*}
\mathscr{Y}_{31} = \{\{(x_1 + x_{2}, a_i), (y_1 + y_{2}, a_j)\} \in \mathscr{Y}_3 : & ~\text{either}~ a_j \in \{x_1, x_{2}\}, ~ a_i \notin \{y_1, y_{2}\} \\ 
& ~\text{or} ~ a_j \notin\{x_1, x_{2}\}, ~ a_i \in \{y_1, y_{2}\}\},
\end{align*} 

and 
\[\mathscr{Y}_{32} = \left\{\{(x_1 + x_{2}, a_i), (y_1 + y_{2}, a_j)\} \in \mathscr{Y}_3 : a_j \in \{x_1, x_{2}\} ~\text{and}~ a_i \in \{y_1, y_{2}\}\right\}.\]

Since $A$ is an  $S_3$-set, $\mathscr{Y}_{31}$ and $\mathscr{Y}_{32}$ partition $\mathscr{Y}_3$.

Obviously, total number of elements of the form $(x_1 + x_{2}, a_{i_g}) \in X_3$ are at most $|G|$. The remaining elements of $X_3$ correspond to distinct elements of $\mathscr{Y}_3$. Thus we get
\[|G| + |\mathscr{Y}_3| \geq |X_3| = \frac{k(k-1)(k-2)}{2}.\]
That is, 
\begin{equation}\label{Y3}
|G| + |\mathscr{Y}_{31}| + |\mathscr{Y}_{32}| \geq \frac{k(k - 1)(k - 2)}{2}.
\end{equation}
We next separately estimate $|\mathscr{Y}_{31}|$ and $|\mathscr{Y}_{32}|$.

Observe that, every set in $\mathscr{Y}_{31}$ admits a unique representation of the form
\[\{(x_1 + a, b),(y_1 + y_2, a)\},\]
where $a, b$ and $x_1$ are pairwise distinct, and $a, b, y_1$ and $y_2$ are pairwise distinct with $y_1 \prec y_2$. We call this the \emph{canonical representation} of this set.
	
Similarly, every set in $\mathscr{Y}_{32}$ admits a unique representation of the form
\[\{(x_1 + a, b),(y_1 + b, a)\},\]	where  $a, b$ and $x_1$ are pairwise distinct, and $a, b$ and $y_1$ are pairwise distinct with $a \prec b$. We call this the \emph{canonical representation} of the set.
	
Consider the set $T_{31}$ defined as:
\begin{align*}
T_{31} = \{(a, b, x_1, y_1, y_2) : & ~ 2a + x_1 = b + y_1 + y_2, ~ \text{with} ~ a, b, x_1  ~\text{are pairwise distinct},\\ &  a, b, y_1, y_2 ~\text{are pairwise distinct and} ~ y_1 \prec y_2\}.
\end{align*}
We define a map $f_{31}:\mathscr{Y}_{31} \to T_{31}$ by
\[f_1\big(\{(x_1 + a,b), (y_1 + y_2,a)\}\big) = (a, b, x_1, y_1, y_2),\]
where $\{(x_1 + a,b), (y_1 + y_2,a)\}$ denotes the canonical representation of the set. Since every set admits a unique canonical representation, the map is well defined. Also, it is straightforward to verify that this map is a bijection.
	
A simple combinatorial argument yields the bound $|T_{31}| \leq 3k(k -1)$.
So we have 
\begin{equation}\label{T31}
|\mathscr{Y}_{31}| \leq 3k(k -1)
\end{equation}
	
Let $\{(x_1 + a, b),(y_1 + b, a)\}$ be a set in $\mathscr{Y}_{31}$. Consider the following set: 
\begin{align*}
T_{32} = \{(a, b, x_1, y_1) : ~ & 2(b - a) = x_1 - y_1,~ \text{where} ~ a, b, x_1  ~\text{are pairwise distinct, and}\\ &  a, b, y_1 ~\text{are pairwise distinct with} ~ a \prec b\}.
\end{align*}
Let 
\begin{align*}
T_{32}' = \{(a, b, x_1, y_1) : & ~ 2(b - a) = x_1 - y_1 = 0,~ \text{where} ~ a, b, x_1   ~\text{are pairwise distinct, and} \\ & a, b, y_1 ~\text{are pairwise distinct with} ~ a \prec b\}.
\end{align*}
Then 
\begin{align*}
T_{32} \setminus T_{32}'  = &\{(a, b, x_1, y_1) : ~  2(b - a) = x_1 - y_1 \neq 0,~ \text{where}~ a, b, x_1  ~\text{are pairwise distinct},\\ & \text{and}  ~ a, b, y_1 ~\text{are pairwise distinct with} ~ a \prec b\}.
\end{align*}
Observe that, if $a ,b, c, d \in A$ such that (i) $a \prec b$ and $c \prec d$ (ii) $b - a = d - c$, and (iii) $2(b - a) = 0$, then $a = c$ and $b = d$. This implies that
\begin{equation}\label{bT32p}
|T_{32}'| \leq |(A \ominus A)[2]|(k - 2). 
\end{equation}
To get an upper bound for $|T_{32} \setminus T_{32}'|$. We define 
\[\mathscr{Y}_3' = \left\{\left\{(a, a_{i_g}), (c, a_j)\right\} \subseteq Y^2_g : g \in 2\cdot(A \ominus A) ~ \text{and} ~ g \neq 0\right\},\]
where $Y^2_g$ and $i_g$ are as defined in the proof of Theorem \ref{s2}.
	
A simple counting yields, 
\begin{equation}\label{bT32c}
|T_{32} \setminus T_{32}'| \leq \frac{k(k-1)}{2} - |(A \ominus A)[2]| + |\mathscr{Y}_3'|
\end{equation}
It is proved in the proof of the Theorem \ref{s2} that 
\[|\mathscr{Y}_{2}| \leq 2k + |(A \ominus A)[2]|.\] Obviously, $\mathscr{Y}_3' \subseteq \mathscr{Y}_{2}$. Therefore,
\begin{equation}\label{by3p}
|\mathscr{Y}_3'| \leq 2k + |(A \ominus A)[2]|.
\end{equation}
Using \eqref{bT32c} and \eqref{by3p}, we get 
\begin{equation}\label{bT32pc}
|T_{32} \setminus T_{32}'| \leq \frac{k(k-1)}{2} + 2k.
\end{equation}
Clearly, 
\[|T_{32}| = |T_{32}'| + |T_{32} \setminus T_{32}'|\]
Using \eqref{bT32p} and \eqref{bT32pc}, we obtain 
\begin{equation}\label{bT32}
|T_{32}| \leq \frac{k(k-1)}{2} + 2k + |(A \ominus A)[2]|(k - 2).
\end{equation}
We define a map $f_{32}:\mathscr{Y}_{32} \to T_{32}$ by
\[f_{32}\big(\{(x_1 + a, b),(y_1 + b, a)\}\big) = (a, b, x_1, y_1),\]	where $\{(x_1 + a, b),(y_1 + b, a)\}$ denotes the canonical representation of the set. Since every set admits a unique canonical representation, the map is well defined. Also, it is straightforward to verify that this map is a bijection. Therefore, by \eqref{bT32}, we obtain
\begin{equation}\label{T32}
|\mathscr{Y}_{32}| \leq \frac{k(k-1)}{2} + 2k + |(A \ominus A)[2]|(k - 2).
\end{equation}
Using  \eqref{T31} and \eqref{T32}, we get 
\begin{equation*}
|\mathscr{Y}_3| \leq 3k(k-1) + \frac{k(k-1)}{2} + 2k + |(A \ominus A)[2]|(k - 2).
\end{equation*}
Hence it follows from \eqref{Y3} that
\[|G| + \frac{7k(k-1)}{2} + 2k + |(A \ominus A)[2]|(k - 3) \geq \frac{k(k-1)(k-2)}{2}.\]
This implies that 
\[|G|  \geq \frac{k(k-1)(k-2)}{2} - \frac{7k(k - 1)}{2} - 2k - |(A \ominus A)[2]|(k - 2).\]
Hence Theorem \ref{s3} follows.	

\end{proof}
		
% ------------------------------------------------------------------------
		
\section*{Acknowledgment}
The research of the first named author is supported by the UGC Fellowship (NTA Ref. No.: 231610040283).

% ------------------------------------------------------------------------

\begin{thebibliography}{99}
		
\bibitem{bajnok2018} B. Bajnok, \emph{Additive Combinatorics: A Menu of Research Problems}, CRC Press, 2018.
					
\bibitem{brouwer1990} A. E. Brouwer, J. B. Shearer, N. J. A. Sloane and W. D. Smith, A new table of constant weight codes, {\it IEEE Trans. Inform. Theory} {\bf 36} (1990) 1334--1380.

\bibitem{graham1980a} R. L. Graham and N. J. A. Sloane, Lower bounds for constant weight codes, {\it IEEE Trans. Inform. Theory} {\bf 26} (1980) 37--43.

\bibitem{graham1980b} R. L. Graham and N. J. A. Sloane, On additive bases and harmonious graphs, {\it SIAM J. Algebraic Discrete Methods} {\bf 1} (1980) 382--404.


\bibitem{guy1994} R. K. Guy, \emph{Unsolved Problems in Number Theory}, 2nd ed., Springer, New York, 1994.
			
\bibitem{pantoja2025} V. C. Guerrero~Pantoja, J. H. Castillo and C. A. Trujillo Solarte, $S_h$-sets and linear codes over $\mathbb{F}_q$, {\it arXiv preprint}, arXiv:2411.19413 (2025).

\bibitem{haanpaa2004} H. Haanp\"a\"a, A. Huima and P. \"Osterg{\aa}rd, Sets in $\mathbb{Z}_n$ with distinct sums of pairs, {\it Discrete Appl. Math.} {\bf 138} (2004), 99--106.
					
\bibitem{haanpaa2007} H. Haanp\"a\"a and P. \"Osterg{\aa}rd, Sets in abelian groups with distinct sums of pairs, {\it J. Number Theory} {\bf 123} (2007), 144--153.		
\end{thebibliography}
\end{document}